\documentclass[11pt,reqno]{amsart}
\usepackage{inputenc}
\usepackage{enumerate}
\usepackage{amssymb}
\usepackage{amsmath}
\usepackage{mathrsfs}
\usepackage{amsthm}
\usepackage{tikz-cd}
\usepackage{mathpazo}
\usepackage{extarrows}
\usepackage{color}
\usepackage{setspace}
\usepackage[colorlinks,linkcolor = blue,anchorcolor = red,urlcolor = blue,citecolor = blue]{hyperref}
\usepackage{graphicx, subfig}
\usepackage[square, comma, sort&compress, numbers]{natbib}

\theoremstyle{definition}
\newtheorem{Def}{Definition}[section]
\newtheorem{Exa}[Def]{Example}
\newtheorem{Rem}[Def]{Remark}
\theoremstyle{plain}
\newtheorem{Thm}[Def]{Theorem}
\newtheorem{Lem}[Def]{Lemma}

\newtheorem{Cor}[Def]{Corollary}

\newtheorem{con}{Conjecture}

\usepackage[left=2.6cm, right=2.6cm, top=3cm, bottom=3cm]{geometry}
\usepackage[all]{xy}
\usepackage{bm}

\usepackage{comment}

\def\IN{\mathbb N}\def\IR{\mathbb R}\def\IC{\mathbb C}

\def\A{\mathcal A}\def\B{\mathcal B}\def\E{\mathcal E}\def\D{\mathcal D}\def\K{\mathcal K}\def\cL{\mathcal L}\def\H{\mathcal H}\def\S{\mathcal S}

\def\supp{\textup{supp}}

\def\prop{\textup{Prop}}

\def\ox{\otimes}
\def\wh{\widehat}
\def\wt{\widetilde}
\def\ox{\otimes}

\def\Ga{\Gamma}

\usepackage{fancyhdr}

\author[J.~Deng]{Jintao Deng}
\address[J.~Deng]{Department of Mathematics, SUNY at Buffalo, NY 14260, USA.}
\email{jintaode@buffalo.edu}

\author[L.~Guo]{Liang Guo}
\address[L.~Guo]{Shanghai Institute for Mathematics and Interdisciplinary Sciences, Shanghai, 200433, P.~R.~China}
\email{liangguo@simis.cn}

\title{Twisted Roe algebras and their $K$-theory}
\date{\today}

\begin{document}

\begin{abstract}
    In this paper, we introduce a notion of twisted Roe algebra and a twisted coarse Baum-Connes conjecture with coefficients. We will study the basic properties of twisted Roe algebras, including a coarse analogue of the imprimitivity theorem for metric spaces with a structure of coarse fibrations. We show that the twisted coarse Baum-Connes conjecture with coefficients holds for a metric space with a coarse fibration structure when the base space and the fiber satisfy the twisted coarse Baum-Connes conjecture with coefficients. As an application, the coarse Baum-Connes conjecture holds for a finitely generated group which is an extension of coarsely embeddable groups. 
\end{abstract}

\maketitle

\tableofcontents

\section{Introduction}

Let $X$ be a metric space with bounded geometry. The coarse Baum-Connes conjecture claims that a certain coarse assembly map
$$
\mu: KX_* (X) \to K_* (C * (X))
$$
is an isomorphism, where $KX_*(X)$ is the coarse $K$-homology of the coarse classifying space of $X$ and $C^*(X)$ is its Roe algebra. This conjecture provides an algorithm to compute higher indices for elliptic operators on non-compact Riemannian manifolds. It has many applications in topology and geometry. In particular, it implies the Novikov conjecture on homotopy invariance of higher signatures, and the Gromov-Lawson-Rosenberg conjecture regarding the nonexistence of positive scalar curvature metrics on closed aspherical manifolds.

The coarse Baum-Connes conjecture has been verified for a large class of metric spaces with bounded geometry, including all coarsely embeddable spaces \cite{Yu2000}, and certain relative expanders \cite{DWY2023}. The purpose of the paper is to study the stability of this conjecture and enlarge the class of metric spaces satisfying the coarse Baum-Connes conjecture in the framework of twisted Roe algebras with coefficients. 

For a metric space with bounded geometry, we first introduce a coarse algebra, then use it to define a twisted version of Roe algebras. Let $Z$ be a metric space with bounded geometry and $\A$ a stable $C^*$-algebra. A coarse $Z$-algebra $\Ga(Z,\A)$ is defined to be some certain $C^*$-subalgebra of $\ell^{\infty}(Z, \A)$. Typical examples of a coarse $Z$-algebra include $\ell^{\infty}(Z,\K)$ and $C_0(Z,\K)$, where $\K$ is the algebra of compact operators on an infinite-dimensional separable Hilbert space, and $C_0(Z,\K)$ is the algebra of all functions from $Z$ to $\K$ vanishing at infinity. 

For a proper metric space $M$ with bounded geometry, we take a subset $Z\subseteq M$ which is coarse equivalent to $M$. Given a coarse algebras $\Ga(Z, \A)$, we can define the twisted Roe algebra $C^*_{\Ga(Z,\A)}(M, \A)$ for the metric space $M$. This twisted Roe algebra can be viewed as a coarse analogue of the crossed product of a group $G$ and a $G$-$C^*$-algebra. Furthermore, we have the coarse imprimitivity theorem.

\begin{Thm}[Coarse imprimitivity theorem]
Let $Z\hookrightarrow X\to Y$ be a coarse $Z$-fibration metric space such that $X$, $Y$ and $Z$ have bounded geometry. For any coarse $Z$-algebra $\Ga(F,\A)$, we have the following Morita equivalence of two $C^*$-algebras:
    $$C_{\Ga_{ind}^Z(X,\A)}^*(X,\A)\sim_{\text{Morita}} C_{\Ga(Z,\A)}^*(Z,\A),$$
where $\Ga_{ind}^Z(X, \A)$ is the induced coarse $X$-algebra associated with the coarse fibration $Z \hookrightarrow X\to Y$ and the $Z$-algebra $\Ga( Z,\A)$.
\end{Thm}

As a typical example, we have
$C^*_{C_0(X, \K)}(X,\K)\sim_{\text{Morita}} \K$, where $C_0(X, \K)$ is the algebra of all $\K$-valued functions on $X$ vanishing at infinity. 
This can be seen as a coarse analogue of $C_0(G,\K)\rtimes_r G\sim_{\text{Morita}} \K$. We also have $C^*_{\ell^{\infty}(X, \K)}(X, \K)\cong C^*(X)$, where $C^*(X)$ is the Roe algebra of $X$.

For a metric space $X$ with bounded geometry and a coarse $X$-algebra $\Ga(X, \A)$, we introduce a coarse assembly map
$$
\mu_{\Ga(X,\A)}: \lim\limits_{d \to \infty} C^*_{L, \Ga(X,\A)}(P_d(X),\A)\to C^*_{\Ga(X,\A)}(X,\A).
$$
The \emph{twisted coarse Baum-Connes conjecture} claims that the coarse assembly map is an isomorphism. In the case when $\Ga(X,\A)=\ell^{\infty}(X,\K)$, the twisted coarse Baum-Connes conjecture is the usual coarse Baum-Connes conjecture, while it is the coarse Baum-Connes conjecture with coefficients in $\A$ when $\Ga(X, \A)=\ell^{\infty}(X, \A)$. 

In \cite{Yu2000}, Yu showed that the coarse Baum-Connes conjecture holds for a metric space which admits a coarse embedding into Hilbert space. We strengthen this result to the twisted Baum-Connes conjecture with coefficients in any coarse algebra.

\begin{Thm}
    Let $X$ be a metric space with bounded geometry which admits a coarse embedding into Hilbert space. If $\Ga(X, \A)$ is any coarse $X$-algebra with a stable fiber $\A$ over $X$, then the twisted coarse Baum-Connes conjecture with coefficients in $\Ga(X, \A)$ holds  for $X$, i.e., the twisted assembly map 
    $$
    \mu_{\Ga(X,\A)}: \lim\limits_{d \to \infty} K_*(C^*_{L,\Ga(X, \A)}(P_d(X), \A))\to K_*(C^*_{\Ga(X, \A)}(X, \A))
    $$
    is an isomorphism. 
\end{Thm}

Let $Z \hookrightarrow X\to Y$ be a coarse fibration. It is possible that $X$ does not admit a coarse embedding into Hilbert space, even though $Z$ and $Y$ are coarsely embeddable into Hilbert space. The following is the main result of this paper. 

\begin{Thm}
Let $Z\hookrightarrow X\xrightarrow{p} Y$ be a coarse fibration with bounded geometry. If both $Z$ and $Y$ satisfy the twisted coarse Baum-Connes conjecture with any coefficients, then the twisted coarse Baum-Connes conjecture with coefficients holds for $X$.
\end{Thm}

We should mention that we take a different approach to prove the twisted coarse Baum-Connes conjecture for the coarse fibration using a coarse analogue of the imprimitivity theorem, other than the Dirac-dual-Dirac method. 

Group extensions provide examples of coarse fibrations. In \cite{AT2019}, Arzhantseva and Tessera showed that admitting a coarse embedding into Hilbert space is not preserved by group extensions. They constructed in \cite{AT2019} an extension of groups $1\to N \to G\to G/N\to 1$ such that $N$ and $G/N$ are coarsely embeddable into Hilbert space, but $G$ is not coarsely embeddable. It is natural to ask if the coarse Baum-Connes conjecture holds for such a group $G$. In \cite{Deng2022}, the first-named author showed that the Novikov conjecture with coefficients holds for $G$, and it follows that the coarse assembly map is injective for $G$. In this paper, we strengthen this result to the following.

\begin{Thm}
Let $1 \to N \to G \to Q \to 1$ be an extension of finitely generated groups. If $N$ and $Q$ are coarsely embeddable into Hilbert spaces, then the coarse Baum-Connes conjecture holds for $G$. 
\end{Thm}

The paper is organized as follows. In Section 2, we introduce the concept of twisted Roe algebras and study their properties, especially the coarse imprimitivity theorem. In Section 3, we introduce the twisted coarse Baum-Connes conjecture and show how our framework of twisted coarse Baum-Connes can be used to prove the usual coarse Baum-Connes conjecture. In particular, we formulate a coarse algebra for a metric space admitting a coarse embedding into Hilbert space. In Section 4, we show that the coarse Baum-Connes conjecture for a metric space $X$ with a coarse fibration structure $Z\hookrightarrow X\to Y$ holds when the fiber $Z$ and the base space $Y$ are coarse embeddable into Hilbert space.

\subsection*{Acknowledgement}
The second-named author wishes to thank Texas A\&M University and SUNY at Buffalo for hospitality during a visit in the spring of 2024.

\section{Twisted Roe algebra with coefficients}
In this section, we will first introduce a coarse algebra for a metric space and use it to define twisted Roe algebra. We will then consider the properties of the twisted Roe algebras. 

\subsection{Coarse algebras}
Let $X$ be a metric space with \emph{bounded geometry} in the sense that for any $R>0$, there exists $N_R\in\IN$ such that
$$\sup_{x\in X}\#B(x,R)\leq N_R.$$
A \emph{partial translation} on $X$ is a bijection $v: D\to R$ from a subset $D\subseteq X$ onto another subset $R \subseteq X$ such that $\sup_{x\in D}d(x,v(x))\leq R$ for some constant $R>0$. Recall that a $C^*$-algebra $\A$ is stable if $\A\otimes \mathcal{K}\cong \A$, where $\mathcal{K}$ is the algebra of compact operators on an infinite-dimensional separable Hilbert space. 

Let $\ell^{\infty}(X,\A)$ be the $C^*$-algebra of all bounded map from $X$ to $\A$. For any map $\xi\in \ell^{\infty}(X, \A)$ and any partial translation $v: D \to R$ on $X$, we define a map $v^*\xi$ from $X$ to $\A$ by
\[
v^*\xi(x) =
\begin{cases}
\xi(v(x)) & \text{~if~} x \in D\\
0 & \text{~if~} x \notin D.
\end{cases}
\]
In the case when $R\cap {\rm Supp}(f)=\emptyset$, we set $v^*f\equiv 0$ on $X$.

\begin{Def} Let $X$ be a metric space with bounded geometry, and $\A$ a stable $C^*$-algebra.
An \emph{algebraic coarse $X$-algebra} is $*$-subalgebra $\Theta(X,\A)$ of $\ell^{\infty}(X,\A)$ satisfying that 
\begin{itemize}
\item[(1)] for any $x\in X$, the set $\{\xi(x)\mid \xi\in\Theta(X,\A)\}$ is dense in $\A$;
\item[(2)] for any partial translation $v: D \to R$ on $X$, and $\xi\in\Theta(X,\A)\subseteq\ell^{\infty}(X,\A)$, 
 $v^*\xi$ is still  an element in $\Theta(X,\A)$.
\end{itemize}
The completion, denoted by $\Ga(X,\A)$, of $\Theta(X,\A)$ in $\ell^{\infty}(X,\A)$ is called a \emph{coarse $X$-algebra}.
\end{Def}

\begin{Exa}
Let $X$ be a metric space with bounded geometry and $\mathcal{A}$ a stable $C^*$-algebra. 
\begin{enumerate}
    \item It is obvious that $\ell^{\infty}(X,\A)$ is a coarse $X$-algebra for any stable algebra $\A$.

\item Let $C_0(X, \mathcal{A})$ be the $C^*$-algebra of all maps from $X$ to $\mathcal{A}$ vanishing at infinity. Then it is obvious that $C_0(X, \mathcal{A})$ is a coarse $X$-algebra. Moreover, every coarse $X$-algebra with fiber $\A$ contains $C_0(X, \A)$ as an ideal.
\end{enumerate}

\end{Exa}
Let $M$ and $X$ be metric spaces. We say that $M$ and $X$ are \emph{coarsely equivalent} if there exists a proper map $f: M\to X$ and non-decreasing functions $\rho_+,\rho_-:\IR_+\to\IR_+$ with $\rho_{\pm}(t)\to\infty$ as $t$ tends to $\infty$, such that
\begin{itemize}
\item[(1)] $\rho_-(d_{M}(x,x'))\leq d_{X}(f(x),f(x'))\leq \rho_+(d_{M}(x,x'))$ for all $x, x'\in M$;
\item[(2)] there exists $r>0$ such that $Z=N_r(f(M))$, where $N_r(f(M))=\{x\in X: \exists~ m \in M,~\text{such~that}~d(x,f(m))\leq r\}$.
\end{itemize}
The map $f$ is called a \emph{coarse equivalence} and $\rho_{\pm}$ are called the \emph{controlling functions} associated with $f$. We say that a family of metric spaces $(M_n)$ is \emph{uniformly coarsely equivalent} to $X$ if there exists a coarse equivalence $f_n: M_n\to X$ for each $n$ such that the family of coarse equivalences $(f_n)_{n\in \mathbb{N}}$ shares the same controlling functions $\rho_{\pm}$ and the same constant $r$ such that $N_r(f_n(M_n))=X$ for all $n \in \mathbb{N}$.

\subsection{Roe algebras and twisted Roe algebras}
Let $M$ be a proper metric space in the sense that the closure of a bounded subset of $M$ is compact. In order to define the twisted Roe algebra for the space $M$, we fix a countable, dense subset $\widehat{M}\subseteq M$. 
\begin{Def}
    Let $T: \widehat{M}\times \widehat{M} \to \A$ be a uniformly bounded map. 
\begin{enumerate}[(1)]
    \item The map $T$ is said to have \emph{finite propagation} if there exists $r>0$ such that $T(x,y)=0$ whenever $d(x,y)>r$ for all $x,y\in \widehat{M}$.
    \item The map $T: \widehat{M}\times \widehat{M} \to \A$ is said to be \emph{locally compact} if the set 
$$
\left\{(x,y)\in (B \times B)\cap \widehat{M}\times \widehat{M}: T_{x,y}\neq 0
\right\}
$$
is finite for any compact subset $B \subseteq M$.
\end{enumerate}
\end{Def}

Consider
\[\H_{\widehat{M},\A}=\ell^2(\widehat{M},\A)=\left\{\xi: \widehat{M}\to \A\ \Big|\ \xi_z\in\A,\ \sum_{z\in \widehat{M}}\xi^*_z\xi_z\text{ converges in norm} \right\}.\]
The space $\H_{X,\A}$ is a right Hilbert $\A$-module as follows. For any $\xi, \eta \in \ell^2(\widehat{M}, \A)$ and $a\in \A$, we define 
$$
\langle \xi, \eta \rangle=\sum_{x}\xi^*_x \cdot \eta_x,
~~~~\text{and}~~~~
\left(\xi\cdot a \right)_x=\xi_x\cdot a.
$$
It is called a \emph{geometric $M$-$\A$-module}. 
We say that a map $T: \widehat{M}\times \widehat{M} \to \A$ is \emph{norm-bounded} if it is bounded when it is viewed as the $\A$-linear map on $\ell^2(\widehat{M},\A)$ by
$$
(T\xi)(x)=\sum_{z \in \widehat{M}}T(x,z)\xi_z
$$
for all $\xi \in \ell^2(\widehat{M}, \A)$.

\begin{Def}
Let $M$ be a proper metric space with bounded geometry, $\widehat{M}\subseteq M$ a countable dense subset, and $\A$ a stable $C^*$-algebra. The Roe algebra $C^*(M,\A)$ of $M$ with coefficients in $\A$ is the $C^*$-subalgebra generated by all locally compact operators on $\ell^2(\widehat{M}, \A)$ with finite propagation.    
\end{Def}

Now, we are ready to impose the twisting structure on Roe algebras.
Let $Z\subseteq M$ be a subspace with bounded geometry satisfying the following
\begin{enumerate}
    \item the inclusion $Z \hookrightarrow M$ is a coarse equivalence;
    \item there exists a uniformly bounded, disjoint Borel cover $\left\{U_z\right\}_{z \in Z}$ of $M$ such that each $U_x$ has non-empty interior and $z\in U_z$ for each $z\in Z$.
\end{enumerate} 
For each $z\in Z$, we consider $B_x=U_z\cap \widehat{M}$. Then the collection $\{B_z\}_{z\in Z}$ is uniformly bounded and $B_z\cap B_{z'}=\emptyset$ for any $z\neq z'$. 
By the stability of the algebra $\A$, we consider a $*$-isomorphism 
$$\psi: \A\otimes \K(\ell^2(\widehat{M}))\to \A.$$
If $T: \widehat{M}\times \widehat{M} \to \A$ is norm-bounded, locally compact operator with finite propagation, then for any $x,y \in Z$, the operator $\chi_{B_x}T\chi_{B_y}$ lies in the algebra $\A\otimes \K(\ell^2(\widehat{M}))$. We obtain a map $\widetilde{T}: Z\times Z \to \A$ by 
$$
\widetilde{T}(x,y)=\psi\left(\chi_{B_x}\cdot \widetilde{T}\cdot \chi_{B_y}\right)
$$
for all $x,y\in Z$. Since each $B_x$ is bounded and $T$ is locally compact, the map $\widetilde{T}$ is well-defined.

Let $v:D \to R$ be a partial translation on $Z$ and let $T: \widehat{M}\times \widehat{M}\to \A$ be a norm-bounded, locally compact map with finite propagation. Then, we can consider a map $T^v:Z\to \A$ by 
\[
T^v(x) =
\begin{cases}
\widetilde{T}(x,v(x)) & \text{~if~} x \in D\\
0 & \text{~if~} x \notin D
\end{cases}
\]

\begin{Def}[Algebraic twisted Roe algebra]\label{def: twisted Roe algebra}
With notations as above, let $\Gamma(Z,\A)$ be a coarse $Z$-algebra with a stable fiber $\A$. The \emph{algebraic twisted Roe algebra of $M$ with twisted coefficients in $\Gamma(Z,\A)$}, denoted by $\IC_{\Gamma(Z,\A)}[M,\A]$, is defined to be the $*$-subalgebra of $B(\ell^2(\widehat{M}, \A))$ all norm-bounded map $T: \widehat{M} \times \widehat{M} \to \A$ satisfying
\begin{itemize}
    \item[(1)] $T$ is locally compact with finite propagation;
    \item[(2)] for any partial translation $v:D \to R$, $T^v\in \Gamma(Z, \A)$.
\end{itemize}
The twisted Roe algebra $C_{\Ga(Z,\A)}^*(M,\A)$ is defined to be the completion of $\IC_{\Ga(Z,\A)}[M,\A]$ under the norm given by the representation on $\H_{\widehat{M},\A}$.
\end{Def}
We would like to point out that the definition of the twisted Roe algebras of $M$ with coefficients in $\Gamma(Z, \mathcal{A})$ does not depend on the choice of dense subset $\widehat{M}$. 

\begin{Exa}
Let $M$ be a metric space with bounded geometry, and $\A$ a stable $C^*$-algebra. One can take the coarse equivalent subset $Z$ to be the whole space $M$. Consider the coarse $Z$-algebra $\Gamma(Z,\A)=\ell^{\infty}(Z,\A)$. 
Then from Definition \ref{def: twisted Roe algebra}, it gives rise to the Roe algebra with coefficients in $\A$ \emph{without twist}. In particular, when $\A=\K$, the twisted Roe algebra is the \emph{Roe algebra} $C^*(M)$. 
\end{Exa}

The following example is a coarse analogue of the imprimitivity theorem in the sense that for a locally compact group $G$, the crossed product $C_0(G)\rtimes G$ is Morita equivalent to the algebra of compact operators.

\begin{Exa}\label{exa: coefficient in c0(X)}
Let $Z$ be a metric space with bounded geometry, and $\K$ the algebra of all compact operators. Consider the coarse $Z$-algebra $\Gamma(Z,\K)=C_0(Z,\K)\subseteq \ell^{\infty}(Z,\K)$ over $Z$ with fiber $\K$. In fact, for any partial translation $v: D \to R$ and $f\in C_0(Z,\K)$, $v^*f$ is still a $C_0$-function. We take $\wh{Z}=Z$ since $Z$ itself is countable. We have that $\H_{Z,\A}=\ell^2(Z,\K)$. 

For any fixed $R>0$, let $\Delta_R=\left\{(x,y):d(x,y)\leq R\right\}$.
Since $Z$ has bounded geometry, for any $R$, we have a decomposition 
of $\Delta_R=\bigsqcup_i^N E^i$ for some $N\in \mathbb{N}$, such that
\begin{enumerate}[(1)]
    \item for each $1\leq i\leq N$, $x\in Z$, there exists at most one $y\in Z$ such that $(x,y)\in E^i$;
    \item for each $1\leq i\leq N$, $y\in Z$, there exists at most one $x\in Z$ such that $(x,y)\in E_i$.
\end{enumerate}

For each $E_i$, we define 
$$D_i=\{x:~\mbox{there~exists}~y\in Z \mbox{~such~that~}(x,y)\in E_i\}$$
and 
$$R_i=\{y:~\mbox{there~exists}~x\in Z \mbox{~such~that~}(x,y)\in E_i\}.$$ 
A partial isometry $v_i: D_i\to R_i$ can be defined by $v^i(x)=y$ if $(x,y)\in E_i$, for all $x\in D_i$.

For any operator $T\in C^*_{C_0(Z,\K)}(Z, \K)$ with propagation less than $R$, we obtain a  decomposition $T=\sum_iT_i$, where
\[
T_i(x, y) =
\begin{cases} 
T(x,y) & \text{if } (x,y)\in E_i \\
0 &  \text{~otherwise}.
\end{cases}
\]
Moreover, we can define $\xi_i(x)=T_i(x,y)$ for all $x\in D_i$ and zero otherwise. Then, we have that $T_i=\xi_iv_i$. By definition, each element $\xi_i\in C_0(Z,\K)$, and we can view it as a compact operator over the $\K$-module $\ell^2(Z, \K)$. This shows that every element in $C^*_{C_0(Z,\K)}(Z, \K)$ with finite propagation is actually compact. On the other hand, $\K(\ell^2(Z))\otimes \K$ is a subalgebra of $C^*_{C_0(Z,\K)}(Z, \K)$. As a result, the canonical inclusion from $\K(\ell^2(Z))\otimes \K$ to $C^*_{C_0(Z,\K)}(Z, \K)$ is an isomorphism, i.e.,
$$C^*_{C_0(Z, \K)}(Z,\K)\cong \K(\ell^2(Z))\otimes \K.$$

\end{Exa}

Example \ref{exa: coefficient in c0(X)} can be viewed as a coarse analogue of the identity $C_0(\Ga,\K)\rtimes \Ga\cong\K(\ell^2(\Ga))\ox\K\cong\K$ for a countable discrete group $\Ga$.

\begin{Exa}\label{exa: uniform algebra}
Let $Z$ be a metric space with bounded geometry. Take $\A=\K$ and $\ell^{\infty}_{\rm ur}(Z,\K)$ be the $C^*$-subalgebra of $\ell^{\infty}(Z,\K)$ generated by all functions $\xi: Z\to\K$ such that each $\xi(x)$ is of finite rank and
$$\sup_{x\in X}\text{rank}(\xi(x))<\infty.$$
One can check that $\ell^{\infty}_{\rm ur}(Z,\K)$ is indeed a coarse $Z$-algebra. In this case, $\IC_{\ell^{\infty}_{\rm ur}(Z,\K)}[Z,\K]$ is a subalgebra of $\IC[Z]$ consists of all $T\in\IC[Z]$ such that $T$ can be approximated by a \emph{uniformly finite rank} operator $S$ in norm, i.e.,
$$\sup_{x,y\in X}{\rm rank}(S(x,y))<\infty.$$
Then $C^*_{\ell^{\infty}_{\rm ur}(Z,\K)}(Z,\K)$ is isomorphic to the \emph{uniform algebra} $UC^*(Z)$ (cf. \cite[Definition 4.1]{SpaWil2013}).

\end{Exa}

\subsection{Coarse imprimitivity}
The following definition is introduced by the second-named author in \cite{GLWZ2023JTA}.

\begin{Def}[Coarse fibration structure]\label{def: coarse fibration structure}
Fix a metric space $Z$. A metric space $X$ is said to have a \emph{coarse $Z$-fibration} over the base space $Y$ if there exists  a surjective map $p: X\to Y$ satisfying the following conditions:\begin{itemize}
\item[(1)] the map $p$ is \emph{bornologous} (or \emph{uniformly expansive}), i.e., for any $R>0$, there exists $S>0$ such that $d_Y(p(x),p(x'))\leq S$ whenever $d_X(x,x')\leq R$;
\item[(2)] the family of fiber spaces $\{Z_y\}$ is uniformly, bijectively coarsely equivalent to $Z$, where $Z_y=p^{-1}(y)\subseteq X$ is the fiber and $\{\phi_y:Z_y\to Z\}_{y\in Y}$ is the family of bijective coarse equivalences;
\item[(3)] for any $R>0$, and each $y \in Y$, there exists a coarse equivalence $\phi_{y,R}:p^{-1}(B_Y(y,R))\to Z$, such that the collection of maps $\{\phi_{y,R}:p^{-1}(B_Y(y,R))\to Z\}$ is a family of uniformly coarse equivalences.
\end{itemize}
Such a space is denoted by $Z\hookrightarrow X\xrightarrow{p}Y$.
\end{Def}

\begin{Exa}
If $X$ and $Y$ are metric spaces, then the product space $X\times Y$ has both a coarse $X$-fibration structure and a coarse $Y$-fibration structure. One can denote these two structures respectively by
$$X\hookrightarrow X\times Y\xrightarrow{p_Y} Y\quad\text{and}\quad Y\hookrightarrow X\times Y\xrightarrow{p_X} X.$$
\end{Exa}

\begin{Exa}\label{exa: coarse fibration structure}
Let $1\to N\to G\to Q\to 1$ be an extension of countable discrete groups. The group $G$ is equipped with a left-invariant proper metric. The group $N$ is equipped with the induced metric of $G$, and $Q$ is equipped with the quotient metric along the quotient map $G \to Q$. One can check that $G$ has a coarse $N$-fibration structure with base space $Q$.
\end{Exa}
Let $X$, $Y$, and $Z$ be metric spaces with bounded geometry such that $Z\hookrightarrow X\to Y$ is a coarse $Z$-fibration. By Definition \ref{def: coarse fibration structure}, there exists a bijective, coarse equivalence $\phi_y: Z_y\to Z$ for each $y\in Y$. Let $\Gamma(Z,\A)$ be a coarse $Z$-algebra. We will use this algebra to define a induced coarse $X$-algebra with fiber $\A$, which can be seen as a coarse analogue of the induced representations for groups, see \cite{Green1978}. For each $y\in Y$, we define
\begin{equation}\label{eq: pullback coefficient}
\Gamma(Z_y,\A)=\left\{
\xi\circ \phi_y: Z_y\to \A\ \Big|\  \xi\in\Ga(Z,\A)
\right\},
\end{equation}
Since each $\phi_y$ is a bijective coarse equivalence, $\Gamma(Z_y,\A)$ is a coarse $Z_y$-algebra. 
Now, we define the \emph{induced coarse $X$-algebra}, denoted by $\Gamma_{ind}^Z(X,\A)$, associated with $\Gamma(Z,\A)$ to be the completion of the following algebra
$$\Theta_{ind}^Z(X,\A)=\bigoplus_{y\in Y}\Gamma(Z_y,\A).$$
Since $X=\bigsqcup_{y\in Y}Z_y$ as a set, we can view every element in $\Theta_{ind}^Z(X,\A)$ as an $\ell^{\infty}$-function from $X$ to $\A$. 

We now prove that $\Gamma_{ind}^{Z}(X,\A)$ is a coarse $X$-algebra. For any partial translation $v:D\to R$ on $X$ and any element $\xi\in\Gamma_{ind}^Z(X,\A)$ with $\supp(\xi)\subseteq D$, by definition, one has that
$$(v^*(\xi))(x)=\xi(v(x)).$$
Since $v$ is a partial translation, there exists $R>0$ such that $d(x,v(x))\leq R$ for any $x\in D$. As $p:X\to Y$ is bornologous, there exists $S$ such that $d_Y(p(x),p(v(x)))\leq S$. Thus if we view $\xi$ as a $C_0$-function from $Y$ to $\bigsqcup_{y\in Y}\Theta(Z_y,\A)$, we define the $\varepsilon$-support of $\xi$ to be
$$\supp_{\varepsilon}(\xi)=\{y\in Y\mid \|\xi_{Z_y}\|_{\infty}\geq\varepsilon\}.$$ 
Then the $\varepsilon$-support of $v_*(\xi)$ is totally contained in $B_Y(\supp_{\varepsilon}(\xi),S)$, as a result, $\|v_*(\xi)|_{Z_y}\|_{\infty}\to 0$ as $y\to\infty$.

On the other hand, for such a partial translation $v: D\to R$, denote $R_y=Z_y\cap R$ for each $y\in Y$. We define
$$D_{y'y}=\{v^{-1}(x)\mid x\in R_y\}\cap Z_{y'}.$$
Let $v_{y'y}$ be the restriction of $v$ on $D_{y'y}$. For each $v_{y'y}$ and $\xi\in\Theta_{ind}^Z(X,\A)$, one can check that $(v_{y'y})^*(\xi|_{D_{y'y}})$ indeed defines an element in $\Gamma(Z_y,\A)$ using condition (3) in Definition \ref{def: coarse fibration structure}. Since $v$ is a partial translation, there are only finitely many $y'$ such that $D_{y'y}$ is nonempty. This proves that $\Gamma_{ind}^Z(X,\A)$ is a coarse $X$-algebra,
We can then define the twisted Roe algebra of $X$ using the coarse algebra $\Gamma^Z_{ind}(X,\A)$ and denote it by $C_{\Ga_{ind}^Z(X,\A)}^*(X,\A)$.

\begin{Thm}[Coarse imprimitivity theorem]\label{thm: coarse imprimitivity}
Let $X$, $Y$ and $Z$ be metric space with bounded geometry such that $Z\hookrightarrow X\to Y$ is a coarse $Z$-fibration. For any coarse $Z$-algebra $\Ga(Z,\A)$, the following two $C^*$-algebras are Morita equivalent:
$$C_{\Ga_{ind}^Z(X,\A)}^*(X,\A)\sim_{\text{Morita}} C_{\Ga(Z,\A)}^*(Z,\A).$$
\end{Thm}

\begin{proof}
Let $\H_{X,\A}=\ell^2(X,\A)=\bigoplus_{y\in Y}\ell^{2}(Z_y,\A)$ be the geometric $X$-module. The geometric module of $Z_y$ is given by $\H_{y,\A}=\ell^2(Z_y,\A)$, therefore we can write $\H_{X,\A}=\bigoplus_{y\in Y}\H_{y,\A}$. Since $\phi_y: Z_y\to Z$ is a coarse equivalence, one can find a covering unitary
\begin{equation}\label{eq: unitary to Z}U_y: \H_{y,\A}\to\H_{Z,\A}=\ell^2(Z,\A),\end{equation}
and $(\phi_y)_*=Ad(U_y)$ induces a $C^*$-isomorphism between $C^*_{\Ga(Z,\A)}(Z,\A)$ and $C^*_{\Ga(Z_y,\A)}(Z_y,\A)$. Define the inclusion
\begin{equation}\label{eq: isometry y to X}W_y: H_{y,\A}\to\H_{X,\A},\end{equation}
Then for each $y\in Y$, the algebra $\IC_{\Ga(Z,\A)}[Z,\A]$ can be viewed as a subalgebra of $\IC_{\Ga_{ind}^Z(X,\A)}[X,\A]$ under the conjugation of the isometry
\begin{equation}\label{eq: isometry Z to X}V_y=W_y\circ U_y^*: \H_{Z,\A}\to \H_{X,\A}.\end{equation}

For any element $y\in Y$, we define 
$$P_y: \H_{X,\A}\to \H_{y,\A}$$
to be the canonical projection from $\H_{X,\A}$ onto $\H_{y,\A}$. It is obvious that the propagation of $P_y$ is exactly $0$. Thus $P_y$ is a multiplier of $C^*_{\Ga^Z_{ind}(X,\A)}(X,\A)$. Moreover, one also has that
$$P_y\cdot\mathbb{C}_{\Ga_{ind}^Z(X,\A)}[X,\A]\cdot P_y=\mathbb{C}_{\Ga(Z_y,\A)}[Z_y,\A]\cong \IC_{\Ga(Z,\A)}[Z,\A].$$
Thus $\IC_{\Ga(Z,\A)}[Z,\A]$ can be seen as a corner of $\IC_{\Ga_{ind}^Z(X,\A)}[X,\A]$. It suffices to show for fixed $y\in Y$ this corner is full, see \cite{BGR1977} or \cite[Definition 1.7.9]{HIT2020}.

For any fixed $R>0$ and any finite subset $K\subseteq Y$, we obtain a coarse fibration $Z\hookrightarrow p^{-1}(K)\to K$. Then we consider the twisted Roe algebra $C^*_{\Gamma(p^{-1}(K), \A)}(p^{-1}(K), \A)$ with coefficients in the reduced algebra $\Gamma(p^{-1}(K), \A)$. Define $\IC_{\Ga_{ind}^Z(p^{-1}(K),\A)}[X,\A]_R$ to be the subalgebra of $\IC_{\Ga_{ind}^Z(X,\A)}[X,\A]$ consisting of all functions $T: X\times X\to\A$ satisfying that\begin{itemize}
\item $\prop(T)\leq R$;
\item for any partial translation $v:D \to R$, $T^v\in\bigoplus_{y\in K}\Ga(Z_y,\A)$.
\end{itemize}
It suffices to show
\begin{equation}\label{eq: inclusion}
\IC_{\Ga_{ind}^Z(p^{-1}(K),\A)}[X,\A]_R\subseteq\IC_{\Ga_{ind}^Z(X,\A)}[X,\A]\cdot P_y\cdot \IC_{\Ga_{ind}^Z(X,\A)}[X,\A],
\end{equation}
where for any $R\geq 0$ and $K\subseteq Y$ a compact subset. Since $p: X\to Y$ is bornologous, there exists $S>0$ such that $d_Y(p(x),p(x'))\leq S$ for any $(x,x')\in X$ with $d(x,x')\leq R$. Denoted by $K'=B(K,S)\subseteq Y$ the $S$-neighborhood of $K$ in $Y$. Then by definition, for any $T\in\IC_{\Ga_{ind}^Z(p^{-1}(K),\A)}[X,\A]_R$, one has that $T(x,y)=0$ whenever $p(x)\notin K$. Since the propagation of $T$ is less than $R$, the support of $T$ is contained in $p^{-1}(K)\times p^{-1}(K')$ for some bounded subset $K'\subseteq Y$. Since $Y$ has bounded geometry, there are only finitely many fibers we need to consider. For any $y_1,y_2\in K'$, we define
$$T_{y_1,y_2}=W^*_{y_1}TW_{y_2},$$
where $W_{y}$ is defined as in \eqref{eq: isometry y to X}. One can see that $T_{y_1,y_2}$ is equal to $P_{y_1}TP_{y_2}$ as an operator form $\H_{y_2,\A}$ to $\H_{y_1,\A}$. Thus $T$ can be written as a finite sum of $\{T_{y_1,y_2}\}_{y_1,y_2\in K'}$. Then it suffices to show such an element $T_{y_1,y_2}$ belongs to the right side in \eqref{eq: inclusion}. Let $U_{y_1}$ be the unitary defined as in \eqref{eq: unitary to Z}, then by the condition (3) in Definition \ref{def: twisted Roe algebra}, one can see that
$$S_{y_1,y_2}=U_{y_1}\cdot T_{y_1,y_2}\cdot U^*_{y_2}\in \IC_{\Ga(Z,\A)}[Z,\A].$$
Without loss of generality, we may assume that $S_{y_1,y_2}$ is positive. Take $S^{1/2}_{y_1,y_2}$ to be its square root. Then
$$V_{y_1}\cdot S^{1/2}_{y_1,y_2}\cdot  V^*_{y}\quad\text{and}\quad V_{y}\cdot S^{1/2}_{y_1,y_2}\cdot V^*_{y_2}$$
defines two elements in $\IC_{\Ga_{ind}^Z(X,\A)}[X,\A]$ (by the condition (3) in Definition \ref{def: twisted Roe algebra}), where $V_{y}$ is defined as in \eqref{eq: isometry Z to X}. Notice that $P_yV_y=V_y$ and $W_yW^*_y=P_y$, one can then compute that
\[\begin{split}
V_{y_1}\cdot S^{1/2}_{y_1,y_2}\cdot  V^*_{y}\cdot P_y\cdot V_{y}\cdot S^{1/2}_{y_1,y_2}\cdot  V^*_{y_2}&=V_{y_1}\cdot S_{y_1,y_2} \cdot V^*_{y_2}\\
&=W_{y_1}U_{y_1}^*\cdot U_{y_1}\cdot T_{y_1,y_2}\cdot U^*_{y_2} \cdot U_{y_2}W^*_{y_2}\\
&=W_{y_1}W^*_{y_1}TW_{y_2}W^*_{y_2}=P_{y_1}TP_{y_2}=T_{y_1,y_2}.
\end{split}\]
This finishes the proof.
\end{proof}

\begin{Rem}
Example \ref{exa: coefficient in c0(X)} can be viewed as a special case of Theorem \ref{thm: coarse imprimitivity}. In fact, we consider the coarse fibration given by $\{pt\}\hookrightarrow X\to X$. The compact operators $\K$ form a coefficient algebra of a single point $\{pt\}$; in this case, the twisted Roe algebra for $\{pt\}$ is $\K$. The induced coarse $X$-algebra is given by $C_0(X,\K)$ by definition. Since the twisted Roe algebra $C_{C_0(X,\K)}^*(X, \K)$ admits a countable approximate unit, then $C_{C_0(X,\K)}^*(X, \K)$ is stably isomorphic to $\K$ by \cite{BGR1977}.
\end{Rem}

\section{The coarse Baum-Connes conjecture and coarsely proper algebras}

In this section, we shall introduce the twisted version of the coarse Baum-Connes conjecture. 

\subsection{The twisted coarse Baum-Connes conjecture}
Let $M$ be a proper metric space and $X\subseteq M$ with bounded geometry such that $X$ is coarse equivalent to $M$. Let us now fix a coarse $X$-algebra.

\begin{Def}[Twisted localization algebra]\label{def: twisted localization algebra}
The \emph{algebraic twisted localization algebra of $M$ with coefficients in $\Gamma(X,\A)$}, denoted by $\mathbb{C}_{L,\Ga(X,\A)}[M,\A]$, is defined to be the $*$-algebra of all uniformly bounded and uniformly continuous functions $g:\IR_+\to\IC_{\Ga(X,\A)}[M,\A]$ such that
$$\prop(g(t))\to 0\quad\text{as}\quad t\to\infty.$$
The \emph{twisted localization algebra} $C^*_{L, \Ga(X, \A)}(M,\A)$ is defined to be the completion of $\mathbb{C}_{L,\Ga(X,\A)}[M,\A]$ under the norm 
$$\|g\|=\sup_{t\in [0,\infty)}\|g(t)\|,$$
for all $g\in \mathbb{C}_{L,\Ga(X,\A)}[M,\A]$.
\end{Def}

There exists a canonical evaluation-at-zero map
$$ev:C^*_{L,\Ga(X,\A)}(M,\A)\to C^*_{\Ga(X,\A)}(M,\A),~~g\mapsto g(0).$$
This is a $*$-homomorphism, thus it induces a $*$-homomorphism
$$ev_*: K_*(C^*_{L,\Ga(X,\A)}(M,\A))\to K_*(C^*_{\Ga(X,\A)}(M,\A)).$$

\begin{Def}[Rips complex]
Let $X$ be a discrete metric space with bounded geometry. For each $d\geq 0$, the \emph{Rips complex} $P_d(X)$ at scale $d$ is defined to be the simplicial polyhedron in which the set of vertices
is $X$, and a finite subset $\{x_0,x_1,\cdots,x_n\}\subseteq X$ spans a simplex if and only if $d(x_i,x_j )\leq d$ for all $0\leq i,j\leq n$.
\end{Def}

There exists a canonical \emph{semi-spherical metric} on the Rips complex defined as in \cite[Definition 7.2.8]{HIT2020}. Under this metric, one can check that $(P_0(X),d_{P_0})$ identifies isometrically with $(X,d)$ and the canonical inclusion $i_d: X\to P_d(X)$ is a coarse equivalence for each $d\geq 0$, see \cite[Proposition 7.2.11]{HIT2020} for details. Fixing a coarse $X$-algebra $\Ga(X, \A)$, we define the twisted localization algebra and the twisted Roe algebra with coefficients in $\Ga(X, \A)$ for the Rips complex associated with $\Theta(X,\A)$. For each $d>0$, we obtain a homomorphism 
$$ev_*: K_*(C^*_{L,\Ga(X,\A)}(P_d(X),\A))\to K_*(C^*_{\Ga(X,\A)}(P_d(X),\A))$$
indiced by the evaluation-at-zero map. 

If $d<d'$, then $P_d(X)$ is included in $P_{d'}(X)$ as a subcomplex via a simplicial map. Passing to the inductive limit, we obtain the \emph{twisted assembly map}
\begin{equation}\label{eq: twisted assembly map}
\mu_{\Ga(X,\A)}:\lim_{d\to\infty}K_*(C^*_{L,\Ga(X,\A)}(P_d(X),\A))\to K_*(C_{\Ga(X,\A)}^*(X,\A)).
\end{equation}

\begin{con}[Twisted coarse Baum-Connes conjecture with coefficients]\label{conj: CBC}
Let $X$ be a metric space with bounded geometry. Then for any coarse $X$-algebra $\Ga(X,\A)$, the twisted assembly map $$\mu_{\Ga(X,\A)}:\lim_{d\to\infty}K_*(C^*_{L,\Ga(X,\A)}(P_d(X),\A))\to K_*(C_{\Ga(X,\A)}^*(X,\A))$$
is an isomorphism.
\end{con}
If we take $\Ga(X,\A)=\ell^{\infty}(X,\A)$, then the twisted assembly map is exactly the coarse Baum-Connes assembly map with coefficients in any $C^*$-algebras. Thus, the Conjecture \ref{conj: CBC} is stronger than the classic coarse Baum-Connes conjecture. 

\subsection{The coarse Green-Julg's Theorem}
In \cite{GLWZ2023JFA}, a notion of \emph{coarsely proper algebra} was introduced to give a conceptual framework for the geometric Dirac-dual-Dirac method. In the rest of this section, we will introduce a coarse analogue of the generalized Green-Julg's theorem.
\begin{Def}\label{def: coarsely proper algebra}
A coarse $X$-algebra $\Ga(X,\A)$ is called \emph{coarsely proper} if there exist a locally compact, Hausdorff space $Z$ such that $\A$ is a $Z$-algebra, a map $f:X\to Z$, and two non-decreasing functions $\rho_+,\rho_-:\IR_+\to\IR_+$ such that\begin{itemize}
\item[(1)] for each $R>0$ and $x\in X$, there exists a neighborhood $O_{x,R}\subseteq Z$ of $f(x)$ such that\begin{itemize}
\item[(i)] $O_{x,R}\subseteq O_{x,R'}$ for all $R<R'$ and $\bigcup_{R>0}O_{x,R}$ is dense in $Z$ for each $x\in X$;
\item[(ii)] $O_{x,R}\subseteq O_{x',\rho_+(d(x,x')+R)}$ for any $x,x'\in X$;
\item[(iii)] $O_{x,\rho_-(R)}\cap O_{x',\rho_-(R)}=\emptyset$ whenever $d(x,x')\geq 2R$;
\end{itemize}
\item[(2)] $\Theta(X,\A)=\bigcup_{R\geq 0}\Theta(X,\A)_R$ is dense in $\Ga(X, \A)$, where
$$\Theta(X,\A)_R=\{\xi\in\Theta(X,\A)\mid \supp(\xi(x))\subseteq O_{x,R}\subseteq Z\}$$
is an ideal of $\Theta(X,\A)$.
\end{itemize}\end{Def}

\begin{Exa}\label{exa: CE}
Let $X$ be a metric space with bounded geometry that admits a coarse embedding into a Hilbert space. Say $f: X\to\H$ is a coarse embedding. Let $\A(\H)$ be the algebra introduced in \cite[Section 5]{Yu2000}, which is a $(\H\times\IR_+)$-algebra. We shall use the notation in \cite[Section 5]{Yu2000}. Define
$$O_{x,R}=\{(\eta,t)\in\H\times\IR_+\mid \|\eta-f(x)\|^2+t^2\leq R^2\}.$$
The coarse $X$-algebra $\Ga(X,\A(\H)\ox\K)_R$ is defined to be the $C^*$-algebra consisting of all functions $\xi: X\to\A(\H)\ox\K$ such that\begin{itemize}
\item[(1)] $\supp(\xi(x))\subseteq O_{x,R}$ for all $x\in X$;
\item[(2)] $\exists N\in\IN$ such that $\xi(x)\in\beta(\A(W_N(x)))$, where $\beta$ is the Bott map as in \cite{HKT1998};
\item[(3)] $\exists K\geq 0$ such that $\|\nabla_{\eta}a_x\|\leq K$ for any $x\in X$, $\eta\in W_N(x)$ with norm $1$ and $a_x\in \A(W_N(x))$ such that $\beta(a_x)=\xi(x)$.
\end{itemize}
Define
$$\Theta(X,\A(\H)\ox\K)=\bigcup_{R\geq 0}\Theta(X,\A(\H)\ox\K)_R.$$
It is obvious that $\Theta(X,\A(\H)\ox\K)$ is an algebraic coarse $X$-algebra. Indeed, for any partial translation $v$ with $\prop(v)\leq R'$ and $\xi\in\Theta(X,\A(\H)\ox\K)_R$, since $f$ is a coarse embedding, there exists some $S>0$ associated with $f$ such that $v^*(\xi)\in \Theta(X,\A(\H)\ox\K)_{R+S}$. Therefore, the completion $\Ga(X, \A(\H)\ox \K)$ of $\Theta(X,\A(\H)\ox\K)$ is a coarse $X$-algebra. 
\end{Exa}

It seems that the coarsely proper algebra defined in Definition \ref{def: coarsely proper algebra} is different from the one in \cite[Definition 3.1]{GLWZ2023JFA}, however, we will point out that they are actually the same. 
\begin{Rem}\label{rem: comparing coarsely proper algebra}
Recall that in \cite{GLWZ2023JFA}, a $C^*$-algebra $\A$ is a coarsely proper algebra associated with the coarse embedding $f: X\to Y$, if there is a locally compact, second countable, Hausdorff space $Z$ such that $\A$ is a $Z$-$C^*$-algebra and a map $\varphi: Y\to Z$ such that for each $x\in X$ and $R>0$, there exists an open set $O_{x, R}\subseteq Z$ such that\begin{itemize}
\item[(1)] $O_{x,R}\cap \varphi(Y)=\varphi(B_Y(f(x),R))$, $O_{x,R}\subseteq O_{x,R'}$ for all $R<R'$ and $\bigcup_{R>0}O_{x,R}$ is dense in $Z$ for each $x\in X$;
\item[(2)] for a sequence $\{B_Y(f(x_i),R)\}_{i\in I}$ with $B_Y(f(x_i),R)\cap B_Y(f(x_j),R)=\emptyset$ for any $i\ne j$, the corresponding sequence of open set $\{O_{x_i,R}\}$ also satisfies that $O_{x_i,R}\cap O_{x_j,R}=\emptyset$ for any $i\ne j$.
\end{itemize}
If a $C^*$-algebra $\A$ is coarsely proper in the sense of \cite{GLWZ2023JFA}, then for each $R>0$, we can define $\Theta(X,\A\ox\K)_R\subseteq \ell^{\infty}(X,\A\ox\K)$ to be the set consisting of all functions $\xi: X\to\A\ox\K$ such that $\xi(x)\subseteq\A_{O_{x,R}}\ox\K$, where $\A_{O_{x,R}}$ is the ideal of $\A$ consisting of all elements $a$ with $\supp(a)\subseteq O_{x,R}$. Then
$$\Theta(X,\A\ox\K)=\bigcup_{R\geq 0}\Theta(X,\A\ox\K)_R$$
is a coarsely proper algebra in the sense of Definition \ref{def: coarsely proper algebra}. One can check that the completion $\Ga(X, \A\ox\K)$ of $\Theta(X,\A\ox\K)$ is a coarse $X$-algebra, because $f$ is a coarse embedding. The second condition in \cite[Definition 3.1]{GLWZ2023JFA} guarantees that $\Theta(X,\A\ox\K)$ satisfies the condition (2) in Definition \ref{def: coarsely proper algebra}.
\end{Rem}

The following result is a coarse analogue of the generalized Green-Julg's theorem which claims that for any locally compact group $G$ and any proper $G$-$C^*$-algebra $A$, the Baum-Connes with coefficients in $A$ holds for $G$.

\begin{Thm}\label{thm: twisted CBC}
Let $\Ga(X,\A)$ be a coarsely proper $X$-algebra. Then the twisted assembly map 
$$\mu_{\Ga(X,\A)}:\lim_{d\to\infty}K_*(C^*_{L,\Ga(X,\A)}(P_d(X),\A))\to K_*(C_{\Ga(X,\A)}^*(X,\A))$$  
is an isomorphism.
\end{Thm}

\begin{proof}
By Remark \ref{rem: comparing coarsely proper algebra}, if $\Ga(X,\A)$ is a coarsely proper algebra, then $\IC_{\Ga(X,\A)}[X,\A]$ is exactly the twisted Roe algebra introduced in \cite{GLWZ2023JFA}. Then it suffices to show that the twisted localization algebra $C^*_{L,\Ga(X,\A)}(P_d(X),\A)$ has the same $K$-theory with the twisted localization algebra defined as in \cite{GLWZ2023JFA}. Notice that in \cite{GLWZ2023JFA}, a function $g$ in the twisted localization algebra should satisfy that $\supp(g(t)(x,y))\subseteq O_{x, R}$ for all $t\in\IR_+$ which makes it smaller than $C^*_{L,\Ga(X,\A)}(X,\A)$ defined in Definition \ref{def: twisted localization algebra}.

By the same arguments in \cite[Proposition 3.7 and 3.11]{Yu1997}, one can show that the $K$-theory of twisted localization algebras is invariant under strongly Lipschitz homotopy equivalence, and it also has Mayer-Vietoris sequence. Moreover, following the arguments in \cite[Theorem 3.2]{Yu1997} of induction on dimensions, we obtain that the K-theory of the two versions of twisted localizations coincide.

For the induction on dimensions of simplicial complexes, we first need to show the two versions of twisted localization algebras have the same $K$-theory for a $0$-dimensional simplicial complex $X$ with bounded geometry. In this case, $X$ is a discrete space. It suffices to prove the canonical inclusion
$$\lim_{R\to\infty}C_{ub}(\IR_+,\Ga(X,\A)_R)\to C_{ub}(\IR_+,\Ga(X,\A))$$
induces an isomorphism on $K$-theory, where $C_{ub}(\IR_+,\Ga(X,\A)_R)$ and $C_{ub}(\IR_+,\Ga(X,\A))$ are the $C^*$-algebras of all uniformly continuous and bounded maps from the non-negative half real lines $\IR_+$ to $\Ga(X,\A)_R$ and $C_{ub}(\IR_+,\Ga(X,\A))$, respectively. Since $\A$ is stable, the algebras $\Ga(X,\A)_R$ and $\Ga(X,\A)$ are quasi-stable for each $R\geq 0$. By \cite[Lemma 12.4.3]{HIT2020}, the evaluation map from $C_{ub}(\IR_+,\Ga(X,\A))$ to $\Ga(X,\A)$ is an isomorphism (similar for $\Ga(X,\A)_R$). Thus it suffices to show
$$\lim_{R\to\infty}\Ga(X,\A)_R\to\Ga(X,\A)$$
induces an isomorphism on $K$-theory. This holds directly from the fact that $\bigcup_{R\geq 0}\Theta(X,\A)_R=\Theta(X,\A)$. 

For the inductive step, we can the decompose an $(n+1)$-dimensional simplicial complex into two parts, following the decomposition in \cite[Theorem 3.2]{Yu1997}. Then, the result follows from the Mayer-Vietoris argument and the five lemma. 
\end{proof}

In \cite{Yu2000}, Yu showed that the coarse Baum-Connes conjecture with coefficients in any $C^*$-algebra holds for a metric space which admits a coarse embedding into Hilbert space. More precisely, let $X$ be the coarsely embeddable space with bounded geometry, and $A$ any $C^*$-algebra. Then the coarse assembly map
$$
\mu: \lim\limits_{d \to \infty}K_*(C_L^*(P_d(X), A))\to \lim\limits_{d \to \infty}K_*(C^*(P_d(X), A))\cong K_*(C^*(X, A))
$$
induced by the evaluation-at-zero map on $K$-theory is an isomorphism. 

Now, we strengthen Yu's result \cite{Yu2000} to the twisted coarse Baum-Connes conjecture with coefficients in any coarse algebra. 

If $X$ coarsely embeds into Hilbert space, we denote $\Ga(X,\A(\H)\ox\K)$ the coarsely proper algebra defined as in Example \ref{exa: CE}. For any coarse $X$-algebra $\Ga(X,\B)$, we define the algebraic coarse $X$-algebra $\Theta(X,\A(\H)\ox\B)$ to be a $*$-subalgebra of $\ell^{\infty}(X,\A(\H)\ox\B)$ generated by the set
$$\left\{\xi\in \ell^{\infty}(X,\A(\H)\ox\B)\mid \exists \eta\in\Ga(X,\A),\zeta\in\Ga(X,\B)\text{ such that }\xi(x)=\eta(x)\ox\zeta(x)\right\}.$$
Let $\Ga(X,\A(\H)\ox\B)$ be the completion of $\Theta(X,\A(\H)\ox\B)$. 
It is obvious that $\Ga(X,\A(\H)\ox\B)$ is a coarsely proper algebra.

\begin{Thm}\label{thm: CBC for CE}
Let $X$ be a metric space with bounded geometry that admits a coarse embedding into Hilbert space. Then the twisted coarse Baum-Connes conjecture with coefficients holds for $X$.
\end{Thm}

\begin{proof}
For any coefficient algebra $\Ga(X,\B)$, we define $\Ga( X,\A(H)\ox\K)$ to be the proper algebra as in Example \ref{exa: CE} and $\Ga( X,\A(H)\ox\B)$ as above. Then, we have the following commuting diagram
$$\begin{tikzcd}
K_{*+1}\left(C^*_{L,\Ga(X,\B)}(X,\B)\right) \arrow[r, "ev_*"] \arrow[d, "\beta"']  & K_{*+1}\left(C^*_{\Ga(X,\B)}(X,\B)\right) \arrow[d, "\beta"]    \\
K_*\left(C^*_{L,\Ga(X,\A(\H)\ox\B)}(X,\A(\H)\ox\B)\right) \arrow[r, "ev_*", "\cong"'] \arrow[d, "\alpha"'] & K_*\left(C^*_{\Ga(X,\A(\H)\ox\B)}(X,\A(\H)\ox\B)\right) \arrow[d, "\alpha"] \\
K_{*+1}\left(C^*_{L,\Ga(X,\B)}(X,\B)\right) \arrow[r, "ev_*"]                      & K_{*+1}\left(C^*_{\Ga(X,\B)}(X,\B)\right)                    
\end{tikzcd}$$
where $\beta$ is induced by the Bott map
$$\prod_{x\in X}\beta_x: \ell^{\infty}(X,\S\ox\K)\to\Ga(X,\A(\H)\ox\K)$$
and $\alpha$ is induced by the Dirac map
$$\prod_{x\in X}\alpha_x: \Ga(X,\A(\H)\ox\K)\to\ell^{\infty}(X,\S\ox\K),$$
following \cite[Section 7]{Yu2000}. By Theorem \ref{thm: twisted CBC}, the middle horizon map is an isomorphism and the compositions of the vertical maps on both sides are isomorphisms by using the same argument with \cite{Yu2000}. This finishes the proof.
\end{proof}

\section{Proof of the main theorem}

In this section, we will prove the following theorem.
\begin{Thm}\label{thm: main theorem}
Let $Z\hookrightarrow X\xrightarrow{p} Y$ be a coarse fibration with bounded geometry. If both $Z$ and $Y$ satisfy the twisted coarse Baum-Connes conjecture with any coefficients, then the twisted coarse Baum-Connes conjecture with any coefficients holds for $X$.
\end{Thm}

Inspired by \cite{Oyono:BC-and-extensions}, we prove this theorem in four steps. First, we shall reduce the twisted coarse Baum-Connes conjecture with coefficients for $X$ to the twisted coarse Baum-Connes conjecture with coefficients for the twisted coarse Baum-Connes conjecture for $X\times Y$ with coefficients in a certain coarse algebra. The spirit of this step is essentially the same as the coarse imprimitivity theorem. Secondly, we shall view the twisted Roe algebra for $X\times Y$ as the twisted Roe algebra of $Y$ with coefficients in the twisted Roe algebra of $X$. Thus this question is reduced to comparing the $K$-theory of the twisted localization algebra of $P_d(X)\times P_{d'}(Y)$ and the twisted localization algebra of $P_d(Y)$ with twisted coefficients in certain twisted Roe algebra of $P_d(X)$. Since both sides are local, we can use the cutting-and-pasting technique to further reduce this question to the case of $1$-skeleton. Fortunately, the question is equivalent to the the twisted coarse Baum-Connes conjecture with coefficients for $Z$ with some certain coefficient when reducing to $1$-skeleton using the coarse imprimitivity theorem. Then the main result follows from the assumption that the twisted coarse Baum-Connes conjecture with coefficients holds for $Z$.

In the rest of this section, we shall carry out these four steps in details.

\noindent{\bf Step 1. Reduction of the question to the twisted coarse Baum-Connes conjecture for $X\times Y$ with certain coefficients.}

Equip $X\times Y$ with the $\ell^1$-product metric, i.e., $d((x,y),(x',y'))=d_X(x,x')+d_Y(y,y')$ for any $(x,y), (x',y')\in X \times Y$. Consider the map
$$p\times id: X\times Y\to Y\times Y,\quad (x,y)\mapsto(p(x),y).$$
The preimage of the diagonal $\{(y,y):y \in Y\}$ under the map $p\times id$ is given by
$$\bigsqcup_{y\in Y}Z_y\times\{y\}.$$
Equip $\bigsqcup_{y\in Y}Z_y\times\{y\}$ with the subspace metric of $X\times Y$. We define a map $g:X\to \bigsqcup_{y\in Y}Z_y\times\{y\}$ by $x\mapsto(x,p(x))\in Z_{p(x)}\times\{p(x)\}$. For any $x,x'\in X$, one has that
$$d_X(x,x')\leq d((x,p(x)),(x',p(x')))\leq d_X(x,x')+d_Y(p(x),p(x')).$$
Since $p$ is a bornologous map, there exists an upper controlling function $\rho_+:\IR_+\to\IR_+$ such that $d_Y(p(x),p(x'))\leq\rho_+(d_X(x,x'))$. This shows that $X$ is coarsely equivalent to $\bigsqcup_{y\in Y}Z_y\times\{y\}$. We then have the following pull-back diagram
\[\begin{tikzcd}
X\times Y \arrow[r, "p\times id"] & Y\times Y         \\
X \arrow[r, "p"] \arrow[u, "g"]  & Y .\arrow[u, hook]
\end{tikzcd}\]

Let $\Ga(X,\A)$ be a coarse $X$-algebra. For each $y\in Y$, the canonical inclusion $i_y: Z_y\to X$ induces a coarse $Z_y$-algebra by the pull-back 
$$\Ga(Z_y,\A)=i_y^*\left(\Ga(X,\A)\right)$$
Then $\Ga(X,\A)$ can be viewed as a subalgebra of $\prod_{y\in Y}\Ga(Z,\A)$. Denoted by $\Delta_{Y,R}\subseteq Y\times Y$ the $R$-diagonal and denoted by $X_R\subseteq X\times Y$ the preimage of $\Delta_{Y,R}$ under $p\times id$. 

For each $R\geq 0$, we define the $C^*$-subalgebra $\B_R\subseteq\ell^{\infty}(X\times Y, \A)$ consisting of all functions $f$ such that\begin{itemize}
\item[(1)] for any $y,z\in Y$, the function $Z_y \to \A$ by $x\mapsto f(x,z)$ for any $x\in Z_y$ is in $\Ga(Z_y,\A)$;
\item[(2)] $f(x,y)=0$ whenever $(x,y)\notin X_R$.
\end{itemize}
We would like to point out that the algebra $\B_R$ is not a coarse algebra. We define 
$$\B=\overline{\bigcup_{R\geq 0}\B_R}.$$ 
It is straightforward to check that $\B$ forms a coarse $(X\times Y)$-algebra since $X$ is a coarse $Z$-fibration.

\begin{Lem}\label{lem: step 1 reduction to product}
The twisted coarse Baum-Connes conjecture with coefficients for $X$ with coefficients in $\Ga(X,\A)$ is equivalent to the twisted coarse Baum-Connes conjecture with coefficients for $X\times Y$ with coefficients in $\B$.
\end{Lem}

\begin{proof}
Denoted by $X_R$ the preimage of $\Delta_{Y,R}$ in $Y\times Y$ under $p\times id$ as above. We have proved that $X_0$ is coarsely equivalent to $X$. Similarly, we can also prove $X_R$ is coarsely equivalent $X$ for any $R>0$. Indeed, one can check that
$$X_R=\bigsqcup_{y\in Y}p^{-1}(B_Y(y,R))\times\{y\}.$$
By the Condition (3) in Definition \ref{def: coarse fibration structure}, there exists a family uniform coarse equivalence
$$\{\phi_{y,R}: p^{-1}(B_Y(y,R))\to Z\}_{y\in Y}.$$
Choose a family of uniform coarse inverse map $\{\psi_y: Z\to Z_y\}_{y\in Y}$ to $\{\phi_y\}$, where $\phi_y$ is given as in Definition \ref{def: coarse fibration structure}. Then we conclude that the map $\varphi_R: X_R\to X_0$ which restricts on $p^{-1}(B_Y(y,R))\times\{y\}$ is equal to $\psi_y\circ \phi_{y,R}$ is a coarse equivalent and the canonical inclusion from $i_R: X_0\to X_R$ is the coarse inverse of this map. Then by coarse invariance of Roe algebras,
$$(i_R)_*: C^*_{\Ga(X,\A)}(X,\A)\to C^*_{\Ga(X,\A)}(X_R,\A)$$
is a $C^*$-isomorphism.

Let $\IC_{\B}[X\times Y,\A]$ be the algebraic twisted Roe algebra with coefficients in $\B$ of $X\times Y$. For any element $T\in \IC_{\B}[X\times Y,\A]$, and any partial translation $v$ on $X\times Y$, we have that $T^v\in B$. For brevity, we assume that for any partial translation $u$ on $X\times Y$, $T^u\in B_{R'}$ for some $R'>0$ depending on the propagation of $T$. Then the support of $T$ is contained in $X_{R'}\times X_{R'}$. Composing with the coarse equivalence $i_{R'}:X \to X_{R'}$, we can view each $T^u$ as an element in $\Ga(X,\A)$ for every partial isometry $u$ on $X_{R'}$.
Therefore, $T$ can be viewed as an element in $\IC_{\Ga(X,\A)}[X_{R'},\A]$. To sum up, we have that
$$C^*_{\B}(X\times Y,\A)=\lim_{R\to\infty}C^*_{\Ga(X,\A)}(X_R,\A).$$
By the continuity of the $K$-theory, we conclude that 
$$K_*(C^*_{\B}(X\times Y,\A))\cong\lim_{R\to\infty}K_*(C^*_{\Ga(X,\A)}(X_R,\A)).$$
As a result, we have that 
$$K_*(C^*_{\B}(X\times Y,\A))\cong K_*(C^*_{\Ga(X,\A)}(X,\A)).$$ 

For localization algebras, it remains to show the canonical inclusion
$$I: \lim_{R\to\infty}C_{L,\Ga(X,\A)}^*(P_d(X_R),\A)\to C^*_{L,\B}(P_d(X\times Y),\A).$$
induces an isomorphism on $K$-theory. The proof is similar to Theorem \ref{thm: twisted CBC}. Since the $K$-theory of both sides are homotopy invariant and admit Mayer-Vietoris sequence. It suffices to prove that the map on K-theory induced by the inclusion is an isomorphism for the case when $X$ is a $0$-dimensional space. For the $0$-dimensional case, the inclusion map between the localization algebras is equivalent to the map
$$I: \lim_{R\to\infty}C_{ub}(\IR_+,\B_R)\to C_{ub}(\IR_+,\B),$$
on the level of $K$-theory. Since $\B_R$ and $\B$ are both quasi-stable, by \cite[Lemma 12.4.3]{HIT2020}, we know that 
$$K_*(C_{ub}(\IR_+,\B_R)) \cong K_*(\B_R),~\mbox{and}~K_*(C_{ub}(\IR_+,\B))\cong K_*(\B).$$
By continuity of $K$-theory, we have that
$$\lim_{R\to\infty}K_*(\B_R) \cong  K_*(\B).$$
This finishes the proof. 
\end{proof}

\noindent{\bf Step 2. Partial localization algebras.}

By Lemma \ref{lem: step 1 reduction to product}, it suffices to prove the coarse Baum-Connes conjecture with coefficient $\B$ holds for $X\times Y$. We shall begin with the definition of a partial version of localization algebras.

Let $M\times N$ be the $\ell^1$-product of $M$ and $N$, $\D$ a coarse $M\times N$-algebra with stable fiber $\A$. Denoted by $\H_{M\times N,\A}=\ell^2(\wh M\times \wh N,\A)$ the geometric $M\times N$-$\A$-module. There are two canonical representations of $C_0(M)$ and $C_0(N)$ on this module $\H_{M\times N,\A}$ since
$$\ell^2(\wh M\times \wh N,\A)\cong\ell^2(\wh M)\ox\ell^2(\wh N,\A)\cong\ell^2(\wh M,\A)\ox\ell^2(\wh N).$$
Thus, for an operator $T\in\cL(\H_{M\times N,\A})$, we can define $\prop_M(T)$ and $\prop_N(T)$ the propagation of $T$ associated with the $C_0(M)$-representation and $C_0(N)$-representation, respectively.

\begin{Def}[Partial localization algebra]
Let $M\times N$ be the $\ell^1$-product of $M$ and $N$, $\D$ a coarse $M\times N$-algebra. The \emph{algebraic partial localization algebra} of $M\times N$ associated with $N$, denoted by $p_N\text{-}\IC_{L,\D}[M\times N,\A]$, is the algebra of all bounded uniform continuous functions
$$g:\IR_+\to\IC_{\D}[M\times N,\A]$$
such that $\prop_N(g(t))$ tends to $0$ as $t$ tends to infinity.
\end{Def}

Since this algebra is local only on the $N$-direction, the homotopy invariance and Mayer-Vietoris sequence only hold for the $N$-direction. We should mention this partial version of the localization algebra is inspired by the $\pi$-localization algebra introduced in \cite{FWY2020}.

For any $d\geq 0$, there is a canonical continuous inclusion
$$P_{d}(X\times Y)\to P_d(X)\times P_d(Y)$$
defined by
$$\sum_{(x,y)\in X\times Y}a_{(x,y)}[(x,y)]\mapsto \left(\sum_{x\in X}\left(\sum_{y\in Y}a_{(x,y)}\right)[x],\sum_{y\in Y}\left(\sum_{x\in X}a_{(x,y)}\right)[y]\right).$$
On the other hand, $P_d(X)$ and $P_{d'}(Y)$ are respectively coarsely equivalent to $X$ and $Y$. By \cite[Lemma 7.2.14]{HIT2020}, there must exist a continuous coarse equivalent
\begin{equation}\label{eq: classifying space for product}P_d(X)\times P_{d'}(Y)\to P_n(X\times Y)\end{equation}
for sufficiently large $n>0$ which restricts to an identity maps from $X\times Y\subseteq P_d(X)\times P_{d'}(Y)$ to $X\times Y\subseteq P_n(X\times Y)$. Moreover, by using \cite[Lemma 7.2.13]{HIT2020}, the composition
$$P_d(X\times Y)\to P_d(X)\times P_d(Y)\to P_n(X\times Y)$$
is homotopy equivalent to the canonical inclusion map
$$P_d(X\times Y)\to P_n(X\times Y).$$
Moreover, the construction of the map \eqref{eq: classifying space for product} is based on a partition of unity of $P_d(X)\times P_{d'}(Y)$. Repeat the argument as in \cite[Lemma 7.2.13]{HIT2020}, one can also prove that the composition
$$P_d(X)\times P_{d'}(Y)\to P_n(X\times Y)\to P_n(X)\times P_n(Y)$$
is also homotopy equivalent to the canonical inclusion
$$P_d(X)\times P_{d'}(Y)\to P_n(X)\times P_n(Y).$$
By passing the parameters to infinity, we conclude the following result.

\begin{Lem}
The map \eqref{eq: classifying space for product} induces an isomorphism on the inductive limit:
$$\lim_{d,d'\to\infty}K_*(C^*_{L,\B}(P_d(X)\times P_{d'}(Y),\A))\to\lim_{n\to\infty}K_*(C^*_{L,\B}(P_n(X\times Y),\A)).\hfill\qed$$
\end{Lem}

We then have the following commuting diagram:
\begin{equation}\label{eq: main diagram}\begin{tikzcd}
\lim\limits_{d,d'\to\infty}K_*(C^*_{L,\B}(P_d(X)\times P_{d'}(Y),\A)) \arrow[r, "ev_*"] \arrow[d, "i_Y"'] & K_*(C^*_{\B}(X\times Y,\A)) \arrow[d, "="] \\
\lim\limits_{d'\to\infty}K_*(p_Y\text{-}C^*_{L,\B}(X\times P_{d'}(Y),\A)) \arrow[r, "ev_*"]                  & K_*(C^*_{\B}(X\times Y,\A))               
\end{tikzcd}\end{equation}
where the horizontal maps are given by the twisted assembly map induced by the evaluation map, and the left vertical map is induced by the canonical inclusion. Here is a short explanation for the map $i_Y$. Notice that $P_d(X)\times P_{d'}(Y)$ is coarsely equivalent to $X\times P_{d'(Y)}$. By coarse invariant, the twisted algebra $C^*_{\B}(P_d(X)\times P_{d'}(Y),\A)$ is isomorphic to $C^*_{\B}(X\times P_{d'}(Y),\A)$ by $Ad_U$ for each $d,d'\geq 0$, where $U$ is a unitary between the geometric module. Since partial localization algebra only shrunk the propagation in the direction of $P_{d'}(Y)$, we can use $C^*_{\B}(X\times P_{d'}(Y),\A)$ to define the partial localization algebra. Thus $i_Y$ is actually induced by the composition
\begin{equation}\label{eq: definition of iY}\begin{tikzcd}
C^*_{L,\B}(P_d(X)\times P_{d'}(Y),\A) \arrow[r, "\text{inclusion}", hook] & p_Y\text{-}C^*_{L,\B}(P_d(X)\times P_{d'}(Y),\A) \arrow[r, "Ad_U"] & p_Y\text{-}C^*_{L,\B}(X\times P_{d'}(Y),\A).
\end{tikzcd}\end{equation}

\begin{Lem}\label{lem: step 2 reduction to iY}
The bottom horizontal map $ev_*$ in \eqref{eq: main diagram} is an isomorphism.
\end{Lem}

\begin{proof}
We shall prove that the map $ev_*$ is essentially a twisted coarse assembly map of $Y$ and this lemma holds as the coarse Baum-Connes conjecture with coefficient holds for $Y$.

For each $y\in Y$, define $\rho_y: Y\to Y\times Y$ by $\rho_y(y')\mapsto (y',y)$. Since $\B$ is a set of functions from $Y\times Y$ to $\prod_{y\in Y}\Theta(Z_y,\A)$. Thus we can define
$$\B_y=\rho_y^*(\B)$$
to be the pull-back to $\B$ under the map $\rho_y$, which is an algebra of functions from $Y$ to $\prod_{y\in Y}\Theta(Z_y,\A)$. Recall the definition of $\B$, one can see that the value of a function in $\B_y$ at $y'$ is taken in $\Theta(Z_{y'},\A)$ and the function will vanish as $y'$ tends to infinity. Thus
\begin{equation}\label{eq: decomposition of By}
\B_y\cong\lim_{S\to\infty}\bigoplus_{y'\in B(y,S)}\Ga(Z_{y'},\A)\cong \bigoplus_{y'\in Y}\Ga(Z_{y'},\A).
\end{equation}
Thus the family $\{\B_y\}$ is actually isomorphic to each other and an element in $\B_y$ can be seen as a function from $\bigsqcup_{y\in Y}Z_y\to \A$, which naturally gives a coarse $X$-algebra structure. For simplicity, we shall denoted
\begin{equation}\label{eq: definition of B0}
\B_0=\bigoplus_{y'\in Y}\Ga(Z_{y'},\A)\quad\text{ and }\quad\B_{y,S}=\bigoplus_{y'\in B(y,S)}\Ga(Z_{y'},\A).
\end{equation}

Define $C^*_{\B_y}(X,\A)$ to be the twisted Roe algebra associated with the coefficient algebra $\B_y$. Define
$$\Theta(Y,C^*_{\B_0}(X,\A)\ox\K)\subseteq\prod_{y\in Y}C^*_{\B_y}(X,\A)\ox\K$$
consists of all functions $\xi:Y\to \prod C^*_{\B_y}(X,\A)\ox\K$ such that\begin{itemize}
\item[(1)] there exists $N>0$ such that $\xi(y)=\sum_{i=1}^NS_{y,i}\ox K_{y,i}$ where $S_{y,i}\in \IC_{\B_y}[X,\A]$ and $K_{y,i}\in\K$;
\item[(2)] there exists $R>0$ and $S>0$ such that $S_{y,i}\in\IC_{\B_{y,S}}[X,\A]_R$ for any $y\in Y$ and $i$, which means that $\prop(S_{y,i})\leq R$ and for any partial translation $v$,
$$S_{y,i}^{v}\in\B_{y,S};$$
\item[(3)] there exists $M>0$ such that $rank(K_{y,i})\leq M$ for any $y\in Y$ and $i$.
\end{itemize}
It is direct to check that $\Theta(Y,C^*(X,\B_0)\ox\K)$ is an algebraic coarse $Y$-algebra.

For any $T\in\IC_{\B}[X\times Y,\A]$ and fixed $y_1,y_2\in Y$, there is a function from $T_{y_1,y_2}: X\times X\to \A$ defined by
$$(x_1,x_2)\mapsto T_{(x_1,y_1),(x_2,y_2)}\in\A.$$
By definition, $T_{(x_1,y_1),(x_2,y_2)}$ is the $((x_1,y_1),(x_2,y_2))$-entry of $T$ viewed as a compact adjointable operator from $\A$ to $\A$. Since $\A$ is assumed to be stable, $T_{(x_1,y_1),(x_2,y_2)}$ determines an element in $\A$. We shall view this function $T_{y_1,y_2}$ as an operator on $\H_{X,\A}=\ell^2(X,\A)$ and the action is given by convolution. In this case, we obtain a family of operators $(T_{y_1,y_2})_{y_1,y_2\in Y}$. One should notice that there are two important facts associated with this family:\begin{itemize}
\item[(1)] the propagation $\prop(T_{y_1,y_2})$ as an operator of $\H_{X,\A}$ is uniformly bounded, i.e., there exists $R>0$ such that $\prop(T_{y_1,y_2})\leq R$ for any $y_1,y_2\in Y$.
\item[(2)] there exists $S>0$ such that
$$T_{y_1,y_2}\in\IC_{\B_{y_1,S}}[X,\A]_R$$
for any $y_1,y_2\in Y$.
\end{itemize}
The first one is because the propagation of $T$ as an operator on $\H_{X\times Y,\A}$ is bounded and the second one is given by the twisted condition. Since $\A$ is stable, we conclude that the Roe algebra $C^*_{\B_y}(X,\A)$ is quasi-stable. Thus, we can identify $C^*_{\B_y}(X,\A)$ with $C^*_{\B_y}(X,\A)\ox p_0$ as a subalgebra of $C^*_{\B_y}(X,\A)\ox\K$, where $p_0$ is a fixed rank $1$ projection. In this case, the map
$$(y_1,y_2)\mapsto T_{y_1,y_2}$$
defines an operator in $\IC_{\Theta(Y,C^*(X,\B_0)\ox\K)}[Y,C^*_{\B_0}(X,\A)\ox\K)]$. This correspondence gives a map
$$\IC[X\times Y,\B]\to\IC_{\Theta(Y,C^*_{\B_0}(X,\A)\ox\K)}[Y,C^*_{\B_0}(X,\A)\ox\K)].$$
This map is not an isomorphism in general, however, by a similar argument as in \cite[Proposition 4.7]{SpaWil2013}, one can show that these two algebras are Morita equivalent to each other by using the uniform rank condition of $\Theta(Y, C^*_{\B_0}(X,\A)\ox\K)$. Thus they have the same $K$-theory. For the localization algebra, one can also check in a similar way that there is a corresponding map
$$p_Y\text{-}C_L^*(X\times P_d(Y),\B)\to C^*_{L,\Ga(Y,C^*_{\B_0}(X,\A)\ox\K)}(P_d(Y),C^*_{\B_0}(X,\A)\ox\K).$$
We shall now show that this map induces an isomorphism on the level of $K$-theory.
Since the $K$-theory of localization algebras have homotopy invariance and Mayer-Vietoris sequence (for partial localization algebra, it has these two properties in the direction of the localized part), it suffices to prove it for the $1$-skeleton of $P_d(Y)$. Thus we have the following diagram
\begin{equation}\label{eq: partial and coefficient}\begin{tikzcd}
p_Y\text{-}C_{L,\B}^*(X\times Y,\A) \arrow[d, "\cong"'] \arrow[r] & C^*_{L,\Ga(Y,C^*_{\B_0}(X,\A)\ox\K)}(P_d(Y),C^*_{\B_0}(X,\A)\ox\K) \arrow[d, "\cong"] \\
C_{ub}(\IR_+,\prod^u_{y\in Y}C^*_{\B_y}(X,\A)) \arrow[r]                     & C_{ub}(\IR_+,\Ga(Y,C^*_{\B_0}(X,\A)\ox\K))                   
\end{tikzcd}\end{equation}
where the symbol $u$ in $\prod^u_{y\in Y}$ means it is a subalgebra of the direct product algebra consists all sequences that can be approximated by sequence with uniformly finite propagation. On the other hand, the right-bottom corner of \eqref{eq: partial and coefficient} is isomorphic to
$$C_{ub}(\IR_+,\Ga(Y,C^*_{\B_0}(X,\A)\ox\K))=\lim_{N\to\infty}C_{ub}\left(\IR_+,\prod^u_{y\in Y}C^*_{\B_y}(X,\A)\right)\ox M_N(\IC).$$
By the stability and continuity of $K$-theory, we conclude that the bottom map in \eqref{eq: partial and coefficient} induces an isomorphism on $K$-theory.

From the above argument, we have shown that the bottom map in \eqref{eq: main diagram} is equivalent to a twisted assembly map for $Y$, by assumption, it is an isomorphism.
\end{proof}

\noindent{\bf Step 3. The cutting-and-pasting technique on localization algebras.}

By Lemma \ref{lem: step 1 reduction to product} and Lemma \ref{lem: step 2 reduction to iY}, to show Theorem \ref{thm: main theorem}, it suffices to show $i_Y$ in \eqref{eq: main diagram} is an isomorphism. In this section, we shall further reduce it to the coarse Baum-Connes conjecture for $X$ with a special coefficient.

\begin{Def}
For each $d\geq 0$, define $A^*(P_d(X))$ be the subalgebra of
$$\prod_{y\in Y} C^*_{\B_y}(P_d(X),\A)$$
consists of all elements $(T_y)_{y\in Y}$ such that there exists $R>0$ and $S>0$ such that
$$T_{y}\in\IC_{\B_{y,S}}[P_d(X)]_{R}$$
for any $y\in Y$, where $\B_{y, S}$ is defined in the previous step, we shall define the propagation of $(T_y)_{y\in Y}$ to be
$$\prop((T_y)_{y\in Y})=\sup_{y\in Y}\prop(T_y).$$
Moreover, we define $A^*_L(P_d(X))$ to be the algebra of all bounded uniformly continuous functions $g$ from $\IR_+$ to $A^*(P_d(X))$ such that
$$\prop(g(t))\to 0\quad\text{as}\quad t\to\infty.$$
\end{Def}

Notice that $A^*(P_d(X))$ are isomorphic to each other for different $d$ since the family $(P_d(X))_{d\geq 0}$ is coarse equivalent to each other. When $d=0$, $A^*(P_0(X))=A^*(X)$. There is a canonical assembly map induced by the evaluation map as follows

\begin{equation}\label{eq: ev map for A}ev_*:\lim_{d\to\infty}A^*_L(P_d(X))\to A^*(X).\end{equation}

\begin{Lem}\label{lem: step 3 cutting and pasting}
To show $i_Y$ in \eqref{eq: main diagram} is an isomorphism, it suffices to show $ev_*$ as in \eqref{eq: ev map for A} is an isomorphism.
\end{Lem}

\begin{proof}
As we discussed before, we have the homotopy invariance and Mayer-Vietoris sequence in the direction of $P_d(Y)$ for $K_*(p_Y\text{-}C^*_{L,\B}(X\times P_{d}(Y),\A))$, by a reduction argument as in \cite[Theorem 3.2]{Yu1997}, it suffices to proof $i_Y$ is an isomorphism for the $1$-skeleton of $Y$. Thus it suffices to prove
$$i_Y: \lim_{d\to\infty}K_*(C^*_{L,\B}(P_d(X)\times Y,\A))\to K_*(p_Y\text{-}C^*_{L,\B}(X\times Y,\A))$$
is an isomorphism. As $Y$ is discrete, the problem can be further reduced. For the left side, it is direct to see that
$$C^*_{L,\B}(P_d(X)\times Y,\A)\cong A^*_L(P_d(X)).$$
On the right side, one has that
$$p_Y\text{-}C^*_{L,\B}(X\times Y,\A)\cong C_{ub}(\IR_+,A^*(X)).$$
Thus it suffices to prove the composition
$$\begin{tikzcd}
A^*_L(P_d(X)) \arrow[r, "\text{inclusion}", hook] & C_{ub}(\IR_+,A^*(P_d(X))) \arrow[r, "\cong"] & C_{ub}(\IR_+,A^*(X))
\end{tikzcd}$$
induces an isomorphism on $K$-theory after passing $d$ to infinity, where the second isomorphism is given by the isomorphic between $A^*(P_d(X))$ and $A^*(X)$. Since $\A$ is stable, $C^*_{\B_y}(X,\A)$ is quasi-stable, so is $A^*(X)$. Thus the evaluation map
$$ev_*: K_*(C_{ub}(\IR_+,A^*(X)))\to K_*(A^*(X))$$
is an isomorphism by \cite[Lemma 12.4.3]{HIT2020}. Thus it suffices to prove the composition
$$\begin{tikzcd}
A^*_L(P_d(X)) \arrow[r, "\text{inclusion}", hook] & C_{ub}(\IR_+,A^*(P_d(X))) \arrow[r, "\cong"] & C_{ub}(\IR_+,A^*(X)) \arrow[r, "ev_*"] & A^*(X).
\end{tikzcd}$$
induces an isomorphism on $K$-theory after passing $d$ to infinity, which is exactly the map $ev_*$ as in \eqref{eq: ev map for A}.
\end{proof}

\noindent{\bf Step 4. Reduction to the twisted coarse Baum-Connes conjecture with coefficients for $Z$.}

For the last step, we shall further reduce the question to the twisted coarse Baum-Connes conjecture with coefficients for $Z$ and finish the proof of Theorem \ref{thm: main theorem}.
\begin{proof}[Proof of Theorem \ref{thm: main theorem}]
By Lemma \ref{lem: step 3 cutting and pasting}, it suffices to prove $ev_*$ in \eqref{eq: ev map for A} is an isomorphism. 

For an element $(T_y)_{y\in Y}\in A^*(X)$, there exist $R>0$ and $S>0$ such that
$$T_{y}\in\IC_{\B_{y,S}}[X]_R$$
for any $y\in Y$, where $\IC_{\B_{y,S}}[X]_R$ is the subspace of $\IC_{\B_{y,S}}[X]$ of all elements with propagation no more than $R$. Write $Z_{y,S}=\bigsqcup_{y'\in B(y,S)}Z_{y'}$ equipped with the subspace metric endowed from $X$. Then, by the definition of $\B_{y,S}$, we conclude that
$$\supp(T_y)\subseteq Z_{y,S}\times Z_{y,S}.$$
Moreover, $\B_{y, S}$ forms a coarse $Z_{y, S}$-algebra, therefore we can define the twisted Roe algebra associated with this coefficient algebra. Then $T_y$ defines an element in $C^*_{\B_{y, S}}(Z_{y, S},\A)$. Thus, $(T_y)_{y\in Y}$ determines an element in $\prod^u_{y\in Y}C^*_{\B_{y, S}}(Z_{y, S},\A)$, where the symbol $``u''$ means it is the subalgebra of the direct product algebra generated by all elements $(T_y)$ such that they can be approximated by sequence of operators with \emph{uniformly} finite propagations. By definition, we have that
$$\lim_{S\to\infty}\prod^u_{y\in Y}C^*_{\B_{y, S}}(Z_{y, S},\A)=A^*(X).$$
Notice that $Z_{y,S}$ is coarsely equivalent to $Z$ by Definition \ref{def: coarse fibration structure} and $\B_{y,S}=\bigoplus_{y'\in B(y, S)}\Theta(Z_{y'},\A)$, we conclude that
$$C^*_{\B_{y,S}}(Z_{y, S},\A)\cong C^*_{\Ga(Z_y,\A)}(Z_y,\A).$$
As the sequence $(Z_{y, S})$ is uniformly coarsely equivalent to $(Z_y)_{y\in Y}$ by definition, the canonical inclusion map from $(Z_y)_{y\in Y}$ to $(Z_{y, S})$ induces an isomorphism
$$\prod^u_{y\in Y}C^*_{\Ga(Z_y,\A)}(Z_y,\A)\xrightarrow{\cong} \prod^u_{y\in Y}C^*_{\B_{y, S}}(Z_{y, S},\A)$$
for each $S\geq 0$. As $S$ passing to infinity, we conclude that
$$K_*\left(\prod^u_{y\in Y}C^*_{\Ga(Z_y,\A)}(Z_y,\A)\right)\cong K_*(A^*(X)).$$

Define
$$\wt Z=\bigsqcup_{y\in Y} Z_y$$
to be the \emph{separated disjoint union} of $(Z_y)_{y\in Y}$, which means the metric $d$ on $\wt Z$ satisfies that $d$ restricts to the original metric on each $Z_y$ and the distance between $Z_y$ and $Z_{y'}$ is infinity whenever $y\ne y'$, see \cite[Definition 12.5.1]{HIT2020}. Notice that $\D=\prod_{y\in Y}\Ga(Z_y,\A)$ is a coarse $\wt Z$-algebra and
$$C^*_{\D}(\wt Z,\A)\cong \prod^u_{y\in Y}C^*_{\Ga(Z_y,\A)}(Z_y,\A).$$
In this case, we can also define the localization algebra $C^*_{L,\D}(P_d(\wt Z),\A)$ for each $d\geq 0$. By a similar argument as in the last paragraph of the proof of Lemma \ref{lem: step 1 reduction to product}, we can also prove that
$$K_*(C^*_{L,\D}(P_d(\wt Z),\A))\cong K_*(A_L(P_d(X))).$$
Thus, it suffices to prove the assembly map induced by the evaluation map
\begin{equation}\label{eq: last step}ev_*: \lim_{d\to\infty}K_*(C^*_{L,\D}(P_d(\wt Z),\A))\to K_*(C^*_{L,\D}(\wt Z,\A))\end{equation}
is an isomorphism. This can be seen as an infinite uniform version of the twisted coarse Baum-Connes conjecture with coefficients of $Z$. Actually, it can be viewed as the twisted coarse Baum-Connes conjecture with coefficients of $Z$ with a specific coefficient. Fix $(\psi_y: Z\to Z_y)_{y\in Y}$ to be a family of the coarse inverse maps to the coarse equivalences $(\phi_y: Z_y\to Z)_{y\in Y}$. We define $\psi^*(\Theta(Z_y,\A))$ to be a coarse $Z$-algebra as following: the function $\xi\in\psi_y^*(\Theta(Z_y,\A))$ if and only if for any $x\in Z_y$,
$$((\psi_y)_*(\xi))(x)=\bigoplus_{z\in \psi_y^{-1}(x)}\xi(z)\in\K(\ell^2(Z)\ox\A)\cong \A$$
gives an element  in $\Theta(Z_y,\A)$ (as we did in \eqref{eq: pullback coefficient}). Define
$$\E\subseteq\ell^{\infty}(Z,\ell^{\infty}(Y,\A)\ox\K)$$
to be the set of all functions $\xi$ such that\begin{itemize}
\item[(1)] there exists $N\in\IN$ such that $\xi(z)=\sum_{i=1}^N\eta_{z,i}\ox K_{z,i}$, where $\eta_{z,i}\in\ell^{\infty}(Y,\A)$ and $K_{z,i}\in\K$;
\item[(2)] for any fixed $y$, the map $z\mapsto\eta_{z,i}(y)$ determines an element in $\psi_y^*(\Theta(Z_y,\A))$;
\item[(3)] there exists $M>0$ such that $rank(K_{z,i})\leq M$ for any $i$ and $z\in Z$.
\end{itemize}
One can check that $\E$ is a coarse $Z$-algebra. For any $y\in Y$, we may write $\E_y=\phi_y^*(\Theta(Z_y,\A))$ to be the restriction of $\E$ at $y$. Then it is direct to check that
$$C^*_{\E}(Z,\ell^{\infty}(Y,\A)\ox\K)\sim_{\text{Morita}} \lim_{n\to\infty}\prod^u_{y\in Y}M_n(C^*_{\E_y}(Z,\A))\cong \prod^u_{y\in Y}C^*_{\Ga(Z_y,\A)}(Z_y,\A)\cong C^*(\wt Z,\D),$$
where the first Morita equivalence is given by the uniform rank condition and the second isomorphism is because $\A$ is stable and $C^*_{\E_y}(Z,\A)$ is quasi-stable. With a similar argument as in the last paragraph in the proof of Lemma \ref{lem: step 2 reduction to iY}, we can show that the map \eqref{eq: last step} is equivalent to the assembly map of $Z$ with coefficient in $\E$. By assumption, the twisted coarse Baum-Connes conjecture with coefficients holds for $Z$, thus the map \eqref{eq: last step} is an isomorphism. This finishes the proof.
\end{proof}

\begin{Rem}
Here is a short remark on the infinite uniform version of the coarse Baum-Connes conjecture (\emph{without coefficient}). For a metric space $X$, if $K_*(C^*(X))$ is \emph{finitely generated}, then the coarse Baum-Connes conjecture for $X$ is equivalent to the coarse Baum-Connes conjecture for the separated disjoint union $\wt X=\bigsqcup_{n\in\IN} X$. Indeed, one should notice that there exists a canonical inclusion
\begin{equation}\label{eq: inclusion of roe algebra}C^*(\wt X)\to\prod_{n\in\IN}C^*(X).\end{equation}
For the $K$-homology side, by \cite[Theorem 6.4.20]{HIT2020}, we have that
$$K_*(C^*_L(P_d(\wt X)))\cong \prod_{n\in\IN}K_*(C^*(P_d(X)))$$
for each $d\geq 0$. Since $C^*(X)$ is quasi-stable, one has that
$$K_*\left(\prod_{n\in\IN}C^*(X)\right)\cong \prod_{n\in\IN}K_*(C^*(X)).$$
Notice that $Ad_{p_n}: C^*(\wt X)\to C^*(X)$ gives a $C^*$-homomorphism, where $p_n$ is the projection on the $n$-th copy of $X$. Thus there is a canonical homomorphism $K_*(C^*(\wt X))\to \prod_{n\in\IN}K_*(C^*(X))$. Hence, it suffices to show this $K$-theory map is an isomorphism. We shall only show it for $K_1$, it is similar for $K_0$. It is direct to see that this map is an injection. For any $a\in K_1(C^*(\wt X))$, if $Ad_{p_n}(a)=0$ for any $n$, then there exists a path which links $Ad_{p_n}(a)$ to $I$ for each $n$. The sequence of the path may not be equi-continuous, but one can use a ``stacking argument'' as in \cite[Proposition 12.6.3]{HIT2020} to exchange space for speed to make this sequence equi-continuous. This sequence gives a path linking $a$ and $I$, this proves the injectivity part.

For the surjectivity part, it suffices to show that for every representative of $K$-theory class on the right side, it is homotopy equivalent to an element that has uniform finite propagation. Say $\{a^{(0)},\cdots,a^{(n)}\}\subseteq K_1(C^*(X))$ is a finite generating set. Since the Baum-Connes conjecture holds for $X$, there exists $d_0\geq 0$ such that one can find elements $\{b^{(0)},\cdots,b^{(n)}\}\subseteq K_1(C^*_L(P_{d_0}(X)))$ such that $ev_*(b^{(k)})=a^{(k)}$ for each $k\in\{0,\cdots,n\}$. One can always find such a $d_0$ since there are only finitely many elements in our consideration. This means that for any $([c_n])_{n\in\IN}\in \prod_{n\in\IN}K_1(C^*(X))$, one can always find $[d_n]\in \prod_{n\in\IN}K_1(C^*_L(P_{d_0}(X)))$ such that $ev_*(d_n)=c_n$. The presentation element $([u_{n}])$ for $(d_{n})$ can be chosen to be equi-continuous by \cite[Theorem 6.4.20]{HIT2020} and has uniform propagation attenuation speed by \cite[Page 33]{GLWZ2023JFA}. For sufficiently large $t$, one has that $(u_{n}(t))\in M_n(C^*(P_d(X)))$ has uniformly finite propagation.
Since $P_d(X)$ is coarse equivalent to $X$, thus $[Ad_U(u_{n}(t))]$ is equal to $c_n$, where $Ad_U$ induces an isomorphism between Roe algebras. This finishes the proof.\qed
\end{Rem}

\section{Some Remarks}

As a direct corollary of Theorem \ref{thm: main theorem} and Theorem \ref{thm: CBC for CE}, we have the following result.

\begin{Cor}\label{cor: ce-by-ce}
Let $Z\hookrightarrow X\to Y$ be a coarse $Z$-fibration with bounded geometry. If $Z$ and $Y$ admit a coarse embedding into Hilbert space, then the twisted coarse Baum-Connes conjecture with coefficients holds for $X$. \qed
\end{Cor}

A metric space $X$ in Corollary \ref{cor: ce-by-ce} is said to admit a \emph{CE-by-CE} coarse fibration structure. For an extension $1\to N\to G\to Q\to 1$ of infinite groups, the quotient map $p: G\to Q$ gives rise to a coarse $N$-fibration to $G$, which gives a \emph{CE-by-CE} coarse fibration structure to $G$. By Corollary \ref{cor: ce-by-ce}, we obtain that the coarse Baum-Connes conjecture holds for the group extensions constructed by G. Arzhantseva and R. Tessera in \cite{AT2019}.

\begin{Thm}\label{cor: finite extension of CE}
Let $G$ be a countable discrete group which is a finite extension of groups that coarsely embeds into Hilbert space. Then the coarse Baum-Connes conjecture (with coefficients) holds for $G$. \qed
\end{Thm}

We also verify the twisted coarse Baum-Connes conjecture with coefficients for the product space.

\begin{Cor}
The twisted coarse Baum-Connes conjecture with coefficients holds for $X$ and $Y$ if and only if the twisted coarse Baum-Connes conjecture with coefficients holds for $X\times Y$.
\end{Cor}

\begin{proof}
For $(\Rightarrow)$, it is a direct conclusion of Corollary \ref{cor: ce-by-ce} and Example \ref{exa: coarse fibration structure}. By Theorem \ref{thm: coarse imprimitivity}, together with a similar argument as in Lemma \ref{lem: step 1 reduction to product}, the twisted coarse Baum-Connes conjecture with coefficients for $X$ is equivalent to the twisted coarse Baum-Connes conjecture for $X\times Y$ with coefficients in certain induced algebra. This proves $(\Leftarrow)$.
\end{proof}

This method can also be used to study \emph{coarse Novikov conjecture}, which claims that the coarse assembly map is injective. 

\begin{con}[Twisted coarse Novikov conjecture with coefficient]
Let $X$ be a metric space with bounded geometry. Then for any coarse $X$-algebra $\Ga(X,\A)$, the twisted assembly map $\mu_{X,\A}$ is injective.
\end{con}

Recall the proof of our main theorem, the twisted coarse Baum-Connes conjecture with coefficients for the base space is used in Step 2, and the twisted coarse Baum-Connes conjecture with coefficients for fiber space is used in Step 3 and Step 4. Notice that in Step 3, we shall need an isomorphism to do the cutting-and-pasting argument, which can not be replaced by injective. But if we only assume the coarse Novikov conjecture with coefficients for the base space, we can still conclude the coarse Novikov conjecture for the whole space $X$, which leads to the following theorem.

\begin{Thm}\label{thm: coarse Novikov}
Let $Z\hookrightarrow X\to Y$ be a coarse $Z$-fibration with bounded geometry. If the twisted coarse Novikov conjecture with coefficients holds for $Y$ and the twisted coarse Baum-Connes conjecture with coefficients holds for $Z$, then the twisted coarse Novikov conjecture with coefficients holds for $X$.
\end{Thm}

In \cite{GLWZ2023JFA}, Z.~Luo, Q.~Wang, Y.~Zhang, and the second author have introduced many examples of coarsely proper algebras with a Bott generator. Using a similar argument with Theorem \ref{thm: CBC for CE}, one can actually show that the coarse Novikov conjecture with coefficient holds for a bounded geometry space $X$ that coarsely embeds into a Hadamard manifold (or a Banach space with Property (H), or an admissible Hilbert-Hadamard space). Therefore, Theorem \ref{thm: coarse Novikov} provides a new approach to the main theorem of \cite[Theorem 1.2]{GLWZ2023JFA}.


\bibliographystyle{alpha}
\bibliography{ref}

\end{document}